\documentclass[12pt,reqno]{amsart}
\usepackage{amsmath,amsthm,amssymb}

\newcommand{\nexteq}{\displaybreak[0]\\ &=}
\newcommand{\nnexteq}{\notag\displaybreak[0]\\ &=}

\newcommand{\refby}[1]{&&\text{(by (\ref{eq:#1}))}}

\newtheorem{thm}{Theorem}[section]
\newtheorem{lem}[thm]{Lemma}
\newtheorem{conj}[thm]{Conjecture}
\theoremstyle{definition}

\newtheorem{dfn}[thm]{Definition}

\newtheorem{proposition}[thm]{Proposition}

\DeclareMathOperator{\ord}{ord}
\DeclareMathOperator{\Tr}{Tr}
\newcommand{\F}{\mathbb{F}}
\newcommand{\cF}{\mathcal{F}}

\title[The Brawley-Carlitz theorem]{A note on the Brawley-Carlitz theorem on irreducibility of composed products of polynomials over finite fields}
\author{Akihiro Munemasa}
\address{Research Center for Pure and Applied Mathematics\\
Graduate School of Information Sciences\\
Tohoku University}
\email{munemasa@math.is.tohoku.ac.jp}

\author{Hiroko Nakamura}
\address{Research Center for Pure and Applied Mathematics\\
Graduate School of Information Sciences\\
Tohoku University}

\date{February 7, 2016}

\begin{document}
\begin{abstract}
We give a new proof of the Brawley-Carlitz theorem on 
irreducibility of the composed products of irreducible polynomials.
Our proof shows that associativity of the binary operation for
the composed product is not necessary. 
We then investigate
binary operations defined by polynomial functions, and
give a sufficient condition in terms of degrees for 
the requirement in the Brawley-Carlitz theorem.
\end{abstract}

\maketitle
\section{Introduction}

For a prime power $q$, we denote by $\F_{q}$ a finite field with $q$ elements. If $m$ and $n$ are relatively prime positive integers, then the composite field of $\F_{q^{m}}$ and $\F_{q^{n}}$ is $\F_{q^{mn}}$. In fact, if $\F_{q^{m}}=\F_{q}(\alpha)$ and $\F_{q^{n}}=\F_{q}(\beta)$, then $\F_{q^{mn}}=\F_{q}(\alpha+\beta)=\F_{q}(\alpha\beta)$. In other words, both $\alpha+\beta$ and $\alpha\beta$ have minimal polynomial of degree $mn$ over $\F_{q}$. 
Brawley and Carlitz generalized this fact by introducing the method of composed products in order to construct irreducible polynomials of large degree from polynomials of lower degree. A basic material of their construction is a binary operation on a subset of $\overline{\F_{q}}$ having certain properties, 
where $\overline{\F_{q}}$ is the algebraic closure of $\F_{q}$. 
Let $G$ be a non-empty subset of $\overline{\F_{q}}$, which is invariant under the Frobenius map $\alpha\mapsto\alpha^q$.
A binary operation $\diamond : G\times G\rightarrow G$ is called a \emph{diamond product} on $G$ if 
\begin{equation}\label{eq:daia}
\sigma(\alpha\diamond\beta)=\sigma(\alpha)\diamond\sigma(\beta) 
\end{equation}
holds for all $\alpha,\beta\in G$.
Let $M_{G}[q,x]$ denote the set of all monic polynomials $f$ in $\F_{q}[x]$ such that $\deg f\geq 1$, and all of the roots of $f$ lie in $G$. 
Let $f(x)=\prod_{i=1}^{m}(x-\alpha_{i})$ and $g(x)=\prod_{i=1}^{n}(x-\beta_{i})$  
be in polynomials in 
$M_{G}[q,x]$, 
where 
$\alpha_{1},\dots,\alpha_{m}, \beta_{1},\dots,\beta_{n}\in G$. 
We define the \emph{composed product} $f\diamond g$ as
\[(f\diamond g)(x)
=\prod_{i=1}^{m}\prod_{j=1}^{n}(x-\alpha_{i}\diamond\beta_{j}).\]

\begin{thm}[{\cite[Theorem 2]{Br-Ca}}]\label{B-C}
Let $\diamond$ be a diamond product on a non-empty subset $G$ of $\overline{\F_{q}}$. Suppose that $(G,\diamond)$ is a group and let $f,g$ be polynomials in $M_{G}[q,x]$ with $\deg f=m$ and $\deg g=n$. Then the composed product $f\diamond g$ is irreducible if and only if $f$ and $g$ are both irreducible with $\gcd(m,n)=1$.
\end{thm}

The purpose of this paper is to give a new proof of Theorem~\ref{B-C} with weaker hypotheses. In order to explain the weakened hypothesis, we need a definition.
For a positive integer $m$, let
\begin{align*}
\cF_{m}(q)&=\{\alpha\in\F_{q^{m}}\mid\F_{q}(\alpha)=
\F_{q^{m}}\}.
\end{align*}
Clearly, 
\begin{equation}\label{honghai}
\cF_{kl}(q)\subset\cF_{l}(q^{k})
\end{equation}
for positive integers $k,l$.

\begin{dfn}
Let $\diamond$ be a diamond product on a subset $G\subset\overline{\F_{q}}$ containing $\cF_{m}(q)\cup\cF_{n}(q)$. We say that $\diamond$ satisfies \emph{weak cancellation} on $\cF_{m}(q)\times\cF_{n}(q)$, if
\begin{align}
\alpha\diamond\beta=\alpha\diamond\beta'\Longrightarrow\beta=\beta',  \label{eq:be}\\
\alpha\diamond\beta=\alpha'\diamond\beta\Longrightarrow\alpha=\alpha' \label{eq:al}
\end{align}
for all $\alpha,\alpha'\in\cF_{m}(q)$ and $\beta,\beta'\in\cF_{n}(q)$. 
\end{dfn}

We will show in Section~\ref{sec:BC} that the conclusion of Theorem~\ref{B-C} holds if $\diamond$ satisfies weak cancellation on $\cF_{m}(q)\times\cF_{n}(q)$. In other words, associativity of the product $\diamond$ is unnecessary. In 
Section~\ref{sec:PF}, we consider a diamond product defined by a polynomial function, and show that such a diamond product satisfies weak cancellation if the degree is small (see Theorem~\ref{jisucan} for details).
In Section~\ref{sec:mirai}, the optimality of the degree bound for weak cancellation is investigated. This leads us to a conjecture on the existence of irreducible polynomials all of whose coefficients except the constant term belong to the prime field.

\section{The Brawley-Carlitz theorem}\label{sec:BC}

Throughout this paper, we let $q$ be a prime power, and $\sigma: \overline{\F_{q}}\rightarrow \overline{\F_{q}}$ denote the Frobenius map $\alpha \mapsto\alpha^q$. For positive integers $k$ and $r$, we denote by $\ord_{k}(r)$ the multiplicative order of $r$ modulo $k$. 
For a nonzero $\alpha\in\overline{\F_{q}}$, we denote by $|\alpha|$ the multiplicative order of $\alpha$. 
Then (see, for example, {\cite[Corollary 2.15]{FF}}), we have,
for $m>1$,
\begin{align*}
\cF_{m}(q)&=
\{\alpha\in\F_{q^{m}}\mid\alpha,\sigma(\alpha),\dots,\sigma^{m-1}(\alpha): \text{pairwise distinct}\}
\nexteq
\{\alpha\in\F_{q^{m}}\mid\{l\in\mathbb{Z}\mid\sigma^{l}(\alpha)=\alpha\}=m\mathbb{Z}\}
\nexteq
\{\alpha\in\F_{q^{m}}\mid\alpha\neq0,\;\ord_{|\alpha |}(q)=m\}.
\end{align*}

Our proof of the Brawley-Carlitz theorem relies on the following lemma 
in group theory. 

\begin{lem}\label{lem:x1}
Let $\Gamma$ be a finite group of order $mn$ having subgroups
$M$ and $N$ of order $m$ and $n$, respectively.
Assume $\Gamma=M\times N$ and $(m,n)=1$.
If $K$ is a subgroup of $\Gamma$, then
$K=(K\cap M)(K\cap N)$.
\end{lem}
\begin{proof}
Since $(m,n)=1$, there exist integers
$r,s$ such that 
$rm+sn=1$.
Let $z\in K$. Since $\Gamma=M\times N$, there exist
$x\in M$ and $y\in N$ such that 
$z=xy$.
Then $z=z^{sn}z^{rm}$ with $z^{sn}=x^{sn}\in K\cap M$, $z^{rm}=y^{rm}\in K\cap N$.
Since $z\in K$ was arbitrary, we conclude
$K\subset(K\cap M)(K\cap N)$. Since the reverse
containment is obvious, we obtain the desired result.
\end{proof}

\begin{thm}
\label{B-CS}
Suppose $G$ is a non-empty subset of $\overline{\F_{q}}$.
Let $\diamond$ be a diamond product on $G$ 
satisfying $(\ref{eq:be})$ and $(\ref{eq:al})$. 
Let $f, g\in M_{G}[q,x]$, 
$\deg f=m$ and $\deg g=n$. Then the following are equivalent: 
\begin{enumerate}
\item $f\diamond g$ is irreducible in $\F_{q}[x]$, 
\item $f$ and $g$ are irreducible in $\F_{q}[x]$, and $\gcd(m,n)=1$. 
\end{enumerate}
\end{thm}

\begin{proof}
(i)$\implies$(ii).
Since $(f\diamond g)(x)$ is irreducible, clearly $f(x)$ and
$g(x)$ are irreducible. Let $\alpha$ and $\beta$ be roots
of $f(x)$ and $g(x)$, respectively. Then $\alpha\diamond\beta$
is a root of $(f\diamond g)(x)$ which is an irreducible polynomial
of degree $mn$. This implies
$\ord_{|\alpha\diamond\beta|}(q)=mn$.
Let $\ell$ be the least common multiple
of $m$ and $n$. Then 
$\sigma^\ell(\alpha\diamond\beta)=
\sigma^\ell(\alpha)\diamond\sigma^\ell(\beta)
=\alpha\diamond\beta$.
Thus $\ord_{|\alpha\diamond\beta|}(q)$ divides $\ell$,
and hence, $\ell=mn$. This implies $\gcd(m,n)=1$.

(ii)$\implies$(i)
Let $\alpha\in\cF_{m}(q)$ and $\beta\in\cF_{n}(q)$ are roots of $f$ and $g$, respectively, so that 
\begin{align*}
f(x)&=\prod_{i=0}^{m-1}(x-\sigma^i(\alpha)),\\
g(x)&=\prod_{i=0}^{n-1}(x-\sigma^i(\beta)).
\end{align*}
The Frobenius automorphism 
$\sigma$ generates
the group $F=\langle\sigma\rangle$ of order
$mn$ acting on $\F_{q^{mn}}$. Moreover, setting
\begin{align}
M&=\langle\sigma^n\rangle=\{g\in F\mid g(\beta)=\beta\},\label{eq:Na}\\
N&=\langle\sigma^m\rangle=\{g\in F\mid g(\alpha)=\alpha\},\label{eq:Mb}
\end{align}
we have $|M|=m$ and $|N|=n$, so 
\begin{equation}\label{eq:MN}
F=M\times N.
\end{equation}
Observe
\begin{align}
M\cdot\alpha&=MN\cdot\alpha
\refby{Na}
\nnexteq
F\cdot\alpha
\refby{MN}
\nnexteq
\{\sigma^i(\alpha)\mid 0\leq i<m\},
\label{eq:Ma}
\end{align}
and similarly
\begin{equation}\label{eq:Nb}
N\cdot\beta=\{\sigma^j(\beta)\mid 0\leq j<n\}.
\end{equation}
Let 
\[K=\{g\in F\mid g(\alpha\diamond\beta)=\alpha\diamond\beta\}.\]
Then 
\begin{equation}\label{eq:GK}
|F\cdot(\alpha\diamond\beta)|=|F:K|.
\end{equation}

We claim $K\cap M=K\cap N=1$. Indeed,
if $g\in K\cap M$, then
\begin{align*}
\alpha\diamond\beta&=g(\alpha\diamond\beta)
\nexteq
g(\alpha)\diamond g(\beta)
\refby{daia}
\nexteq
g(\alpha)\diamond \beta,
\end{align*}
so $\alpha=g(\alpha)$ by (\ref{eq:al}). This implies $g\in N$. 
Since $g\in M$ and $M\cap N=1$, we conclude $g=1$.
This proves $K\cap M=1$. Similarly, we can prove $K\cap N=1$
using (\ref{eq:be}).

Now, by Lemma~\ref{lem:x1}, we obtain $K=1$. This implies
$|F\cdot(\alpha\diamond\beta)|=|F|$ by (\ref{eq:GK}). Therefore,
the polynomial
\begin{align*}
(f\diamond g)(x)&=
\prod_{i=0}^{m-1}\prod_{j=0}^{n-1}(x-\sigma^i(\alpha)\diamond\sigma^j(\beta))
\nexteq
\prod_{g\in M}\prod_{h\in N}(x-g(\alpha)\diamond h(\beta))
&&\text{(by (\ref{eq:Ma}), (\ref{eq:Nb}))}
\nexteq
\prod_{g\in M}\prod_{h\in N}(x-gh(\alpha)\diamond gh(\beta))
\nexteq
\prod_{g\in F}(x-g(\alpha)\diamond g(\beta))
\refby{MN}
\nexteq
\prod_{g\in F}(x-g(\alpha\diamond\beta))
\refby{daia}
\end{align*}
is irreducible over $\F_q$.
\end{proof}

\section{Diamond products defined by polynomial functions}\label{sec:PF}

Stichtenoth \cite{Stichten} classified associative diamond products defined by a polynomial function, under a certain condition. 
As we have seen in the previous section, associativity is irrelevant for the Brawley-Carlitz theorem. This prompts us to classify diamond products satisfying weak cancellation  instead. In this section, we consider diamond products defined by a polynomial function, and give a sufficient condition in terms of degrees in order that the associated diamond product satisfies weak cancellation. 
It turns out that, in general, a wider class of polynomials than those classified in  \cite{Stichten} can be used as a diamond product.

Let $m$ be a positive integer, and let $\psi:\cF_{m}(q)\rightarrow\cF_{m}(q)$ be a function. We say that $\psi$ satisfies the \emph{restricted injectivity} on $\cF_{m}(q)$ if, for all $\alpha\in\cF_{m}(q)$ and $k\in\mathbb{Z}$, 
\begin{equation}
\psi(\alpha)=\psi(\sigma^{k}(\alpha))\Longrightarrow\alpha=\sigma^{k}(\alpha).
\label{eq:(1)}
\end{equation}
If $\psi:\cF_{m}(q)\rightarrow\overline{\F_{q}}$ is a function taking values in  $\F_{q^{m}}$ such that $\psi$ commutes with $\sigma$, then 
$\psi(\sigma^{k}(\alpha))=\sigma^{k}(\psi(\alpha))$. Thus, (\ref{eq:(1)}) is equivalent to 
\begin{equation}\label{eq:douchi}
\psi(\alpha)\in\cF_{m}(q).
\end{equation}
In particular, this equivalence holds when $\psi$ is a polynomial function with coefficients in $\F_{q}$.

\begin{lem}\label{degjisu}
Let $\psi(x)\in\F_{q}[x]$ be a polynomial with $\deg\psi\geq 1$.  
Then for $\alpha\in\overline{\F_{q}}$, 
\[
\deg\psi\geq[\F_{q}(\alpha):\F_{q}(\psi(\alpha))].
\]
\end{lem}

\begin{proof}
Let $\psi_{0}(x)=\psi(x)-\psi(\alpha)\in\F_{q}(\psi(\alpha))[x]$. Then $\psi_{0}(\alpha)=0$, so $\psi_{0}$ is divisible by the minimal polynomial of $\alpha$ over $\F_{q}(\psi(\alpha))$. This implies 
\begin{align*}
[\F_{q}(\alpha):\F_{q}(\psi(\alpha))]&\leq\deg\psi_{0}
\nexteq\deg\psi. 
\end{align*}
\end{proof}

\begin{lem}
Let $m>1$ be an integer, and let $m_{1}$ be the smallest prime divisor of $m$. 
If $\psi(x)\in\F_{q}[x]$ is a monic polynomial with $0<\deg\psi<m_{1}$, then the function defined by $\psi$ satisfies the restricted injectivity on $\cF_{m}(q)$. 
\end{lem}

\begin{proof}
For $\alpha\in\cF_{m}(q)$, we have 
\begin{align*}
m_{1}&>\deg\psi
\\&\geq
[\F_{q}(\alpha):\F_{q}(\psi(\alpha))]
&&\text{(by Lemma~\ref{degjisu})}
\nexteq
[\F_{q^{m}}:\F_{q}(\psi(\alpha))]. 
\end{align*}
Since $[\F_{q^{m}}:\F_{q}(\psi(\alpha))]$ is a divisor of $m$ and $m_{1}$ is the smallest prime divisor of $m$, we conclude that $[\F_{q^{m}}:\F_{q}(\psi(\alpha))]=1$, that is, $\F_{q}(\psi(\alpha))=\F_{q^{m}}$. This implies (\ref{eq:douchi}).  
\end{proof}

\begin{lem}\label{dokuritu}
Let $m_{1}$ be the smallest prime divisor of a positive integer $m>1$, and let $k$ be an integer not divisible by $m$. Then for $\alpha\in\cF_{m}(q)$,
$\alpha-\sigma^{k}(\alpha),\alpha^{2}-\sigma^{k}(\alpha^{2}),\dots,
\alpha^{m_{1}-1}-\sigma^{k}(\alpha^{m_{1}-1})$ are linearly independent over $\F_{q}$.
\end{lem}
\begin{proof}
Suppose $\alpha-\sigma^{k}(\alpha), \alpha^{2}-\sigma^{k}(\alpha^{2}), \dots, 
\alpha^{m_{1}-1}-\sigma^{k}(\alpha^{m_{1}-1})$ are linearly dependent. 
Then there exist $a_{1},\dots,a_{m_{1}-1}\in\F_{q}$, 
$(a_{1},\dots,a_{m_{1}-1})\neq(0,\dots,0)$, 
and $$\sum_{i=1}^{m_{1}-1}a_{i}(\alpha^{i}-\sigma^{k}(\alpha^{i}))=0.$$ 
Let 
\begin{align*}
a_0&=\sum_{i=1}^{m_{1}-1}a_{i}\alpha^{i}\in\F_{q^{m}}, \\
f(x)&=\sum_{i=1}^{m_{1}-1}a_{i}x^{i}-a_{0}\in\F_{q}(a_{0})[x]. 
\end{align*}
Then $\sigma^k(a_{0})=a_0$, so $f\in\F_{q^{\gcd(k,m)}}[x]$. 
Since $f(\alpha)=0$, $f$ is divisible by the minimal polynomial of $\alpha$ over $\F_{q^{\gcd(k,m)}}$. 
This implies 
\begin{align*}
m_{1}&>
\deg f
\\&\geq 
[\F_{q^{\gcd(k,m)}}(\alpha):\F_{q^{\gcd(k,m)}}]
\nexteq
[\F_{q^{m}}:\F_{q^{\gcd(k,m)}}]
\nexteq
\frac{m}{\gcd(k,m)}.
\end{align*}
Since $m_{1}$ is the smallest prime divisor of $m$, we obtain 
$\gcd(k,m)=m$, that is, $m\mid k$. This is a contradiction. 
\end{proof}

\begin{lem}\label{lem:e3b}
If $m$ and $n$ are relatively prime positive integers, then
$\cF_n(q)\subset\cF_n(q^m)$.
In particular, if $\alpha\in\cF_{m}(q)$, $\beta\in\cF_{n}(q)$, $k\in\mathbb{Z}$ and 
\begin{equation*}
\varphi(x,y)=\sum_{i=0}^{n-1}\psi_{i}(x)y^{i}\in\F_{q}[x,y]
\end{equation*}
satisfy $\varphi(\sigma^{k}(\alpha),\beta)=\varphi(\alpha,\beta)$, then $\psi_{i}(\sigma^{k}(\alpha))=\psi_{i}(\alpha)$ for $0\leq i\leq n-1$.
\end{lem}
\begin{proof}
The first part is immediate from \cite[Corollary 3.47]{FF}.
Since $\varphi(\sigma^{k}(\alpha),\beta)=\varphi(\alpha,\beta)$, we have 
\begin{equation*}
\sum_{i=0}^{n-1}(\psi_{i}(\sigma^{k}(\alpha))-\psi_{i}(\alpha))\beta^{i}=0. 
\end{equation*}
Since $\beta\in\cF_{n}(q)\subset\cF_{n}(q^{m})$ and $\psi_{i}(\sigma^{k}(\alpha))-\psi_{i}(\alpha)\in\F_{q^{m}}$, linear independence of $1,\beta,\dots, \beta^{n-1}$ over $\F_{q^{m}}$ shows $\psi_{i}(\sigma^{k}(\alpha))=\psi_{i}(\alpha)$ for $0\leq i\leq n-1$.
\end{proof}

\begin{thm}\label{jisucan}
Let $q$ be a prime power, and let $m,n>1$ be relatively prime positive integers.
Suppose $m_{1}$ is the smallest prime divisor of $m$, 
$n_{1}$ is the smallest prime divisor of $n$.
Let $\varphi(x,y)\in\F_{q}[x,y]$ be a polynomial with  
$0<\deg_{x}\varphi<m_{1}$ and $0<\deg_{y}\varphi<n_{1}$. Then 
the diamond product on $\overline{\F_{q}}$ defined by $\varphi$ satisfies weak cancellation on $\cF_{m}(q)\times\cF_{n}(q)$.  
\end{thm}
\begin{proof}
We need to show
\begin{align}
\varphi(\alpha,\beta)&=\varphi(\sigma^{k}(\alpha),\beta)
\Longrightarrow\alpha=\sigma^{k}(\alpha),
\label{eq:L}\\
\varphi(\alpha,\beta)&=\varphi(\alpha,\sigma^{k}(\beta))
\Longrightarrow\beta=\sigma^{k}(\beta).
\label{eq:R}
\end{align}
It suffices to show only (\ref{eq:L}), as the proof of (\ref{eq:R}) is similar. 
Suppose $\alpha\in\cF_{m}(q)$, $\beta\in\cF_{n}(q)$, $k\in\mathbb{Z}$. 
Let  
\begin{equation}\label{eq:fai}
\varphi(x,y)=\sum_{i=0}^{n_{1}-1}\psi_{i}(x)y^{i},
\end{equation}
\begin{equation}\label{eq:pusai}
\psi_{i}(x)=\sum_{j=0}^{m_{1}-1}a_{ij}x^{j} \quad (0\leq i\leq n_{1}-1).
\end{equation}
If  
$\varphi(\alpha,\beta)=\varphi(\sigma^{k}(\alpha),\beta)$, then, by Lemma~\ref{lem:e3b}, 
$\psi_{i}(\alpha)-\psi_{i}(\sigma^{k}(\alpha))=0$ for $0\leq i\leq n_{1}-1$. 
This implies  
\begin{equation*}
\sum_{j=1}^{m_{1}-1}a_{ij}(\alpha^{j}-\sigma^{k}(\alpha^j))=0 \quad (0\leq i\leq n_{1}-1).   
\end{equation*}
If $\alpha\neq\sigma^k(\alpha)$, then $k$ is not divisible by $m$. Then by 
Lemma~\ref{dokuritu}, we obtain 
$a_{ij}=0$ for $0\leq i\leq n_{1}-1$ and $1\leq j\leq m_{1}-1$.
This implies
$\deg_{x}\varphi=0$, which contradicts the assumption. Therefore,  
$\alpha=\sigma^k(\alpha)$.
\end{proof}

\section{Irreducible polynomials all of whose coefficients except the constant term belong to the prime field}\label{sec:mirai}

In this section, we show that the hypotheses $\deg_{x}\varphi<m_{1}$ and $\deg_{y}\varphi<n_{1}$ in Theorem~\ref{jisucan} are necessary.
We believe that these upper bounds cannot be relaxed for any prime power $q$ and relatively prime positive integers $m$ and $n$. This leads to a conjecture on the existence of irreducible polynomials all of whose coefficients except the constant term belong to the prime field.

\begin{proposition}\label{lem:e3}
Let $m$ and $n$ be relatively prime integers with $m,n>1$.
Let $m_1$ and $n_1$ be the smallest prime divisor of $m$ and $n$,
respectively. 
Then the following are equivalent:
\begin{enumerate}
\item there exists a polynomial
$\varphi(x,y)\in\F_q[x,y]$ with
$\deg_x\varphi=m_1$, $0<\deg_y\varphi<n_1$,
such that 
$\sigma^k(\alpha)\neq\alpha$ and
$\varphi(\sigma^k(\alpha),\beta)=\varphi(\alpha,\beta)$
for some
$\alpha\in\cF_m(q)$ and $\beta\in\cF_n(q)$,
\item there exists a polynomial
$\psi(x)\in\F_q[x]$ with $\deg\psi=m_1$ which fails to
satisfy the restricted injectivity on $\cF_m(q)$,
\item 
$\cF_{m/m_1}(q)\cap\{\alpha^{m_1}+\sum_{i=1}^{m_1-1}c_i\alpha^i\mid
\alpha\in\cF_m(q),\;c_1,\dots,c_{m_1-1}\in\F_q\}\neq\emptyset$, 
\item 
there exists a monic irreducible polynomial $f(x)\in\F_{q^{m/m_{1}}}[x]$ of degree $m_{1}$ such that $f(x)-f(0)\in\F_{q}[x]$ and $f(0)\in\cF_{m/m_{1}}(q)$. 
\end{enumerate}
\end{proposition}
\begin{proof}
(i)$\implies$(ii). Let $\varphi(x)$ be as in (\ref{eq:fai}), 
where $\psi_i(x)\in\F_q[x]$ for $0\leq i\leq n_1-1$. 
Then by Lemma~\ref{lem:e3b}, 
$\psi_i(\sigma^k(\alpha))=\psi_i(\alpha)$ for $0\leq i\leq n_1-1$.
By the assumption,
there exists $i\in\{0,1,\dots,n_1-1\}$ such that $\deg\psi_i=m_1$, 
and this $\psi_i$ fails to satisfy the restricted injectivity
on $\cF_m(q)$.

(ii)$\implies$(iii).
We may assume without loss of generality that $\psi$ is monic.
Replacing $\psi(x)$ by $\psi(x)-\psi(0)$, we may further assume that 
$\psi(0)=0$.
By the assumption, there exists $\alpha\in\cF_m(q)$ and $k\in\mathbb{Z}$ such that
$\sigma^k(\alpha)\neq\alpha$ and 
$\psi(\sigma^k(\alpha))=\psi(\alpha)$.
Since $\psi(x)\in\F_q[x]$, the latter implies
$\sigma^k(\psi(\alpha))=\psi(\alpha)$.
Thus $\psi(\alpha)\in\F_{q^{\gcd(k,m)}}$.
Since $\sigma^k(\alpha)\neq\alpha$, $k$ is not a multiple of $m$.
This implies that
$\F_{q^{\gcd(k,m)}}$ is a proper subfield of $\F_{q^m}$.
Therefore, there exists a divisor $d>1$ of $m$ such that
$\psi(\alpha)\in\F_{q^{m/d}}$.
By Lemma~\ref{degjisu}, we have
\begin{align*}
m_1&\geq [\F_q(\alpha):\F_q(\psi(\alpha))]
\\&\geq
[\F_{q^{m}}:\F_{q^{m/d}}]
\nexteq
d.
\end{align*}
Since $m_1$ is the smallest prime divisor of $m$, we obtain $m_1=d$.
This forces $\F_q(\psi(\alpha))=\F_{q^{m/m_1}}$, and hence
$\psi(\alpha)\in\cF_{m/m_1}(q)$.

(iii)$\implies$(iv).
Suppose $\alpha\in\cF_m(q)$, $c_1,\dots,c_{m_1-1}\in\F_q$,
and 
\[c_{0}=\alpha^{m_1}+\sum_{i=1}^{m_1-1}c_i\alpha^i\in\cF_{m/m_1}(q).\]
Define 
$$f(x)=x^{m_{1}}+\sum_{i=1}^{m_{1}-1}c_{i}x^{i}-c_{0}\in\F_{q^{m/m_{1}}}[x].$$
Then $f(x)-f(0)\in\F_{q}[x]$ and $f(0)=-c_{0}\in\cF_{m/m_{1}}(q)$. We claim $f(x)$ is irreducible in $\F_{q^{m/m_{1}}}[x]$. 
Indeed, since $f(\alpha)=0$, $f(x)$ is divisible by the minimal 
polynomial of $\alpha$ over $\F_{q^{m/m_{1}}}$.
On the other hand, since 
$\F_{q^{m/m_{1}}}(\alpha)\supset\F_{q}(\alpha)=\F_{q^{m}}$,
the minimal polynomial of $\alpha$ over $\F_{q^{m/m_{1}}}$
has degree at least $[\F_{q^m}:\F_{q^{m/m_1}}]=m_1=\deg f$. 
Therefore, $f(x)$ is the
minimal polynomial of $\alpha$ over $\F_{q^{m/m_{1}}}$,
and hence is irreducible $\F_{q^{m/m_{1}}}[x]$. 

(iv)$\implies$(i). 
Define 
\begin{align*}
k&=\frac{m}{m_1},\\
\varphi(x,y)&=(f(x)-f(0))y\in\F_{q}[x,y].
\end{align*}
Then, $\deg_{x}\varphi=m_{1}$, $\deg_{y}\varphi=1$.
Let $\alpha$ be a root of $f(x)$. 
Since $f(0)=-(f(\alpha)-f(0))\in\F_q(\alpha)$, we have
$\F_{q}(\alpha)=\F_{q}(f(0),\alpha)=\F_{q^{m/m_{1}}}(\alpha)
=\F_{q^m}$. 
Thus $\alpha\in\cF_{m}(q)$. 
Moreover, for an arbitrary $\beta\in\cF_n(q)$, we have
\begin{align*}
\varphi(\sigma^k(\alpha),\beta)&=
\sigma^k(f(\alpha)-f(0))\beta
\nexteq
-\sigma^k(f(0))\beta
\nexteq
-f(0)\beta 
\nexteq
(f(\alpha)-f(0))\beta 
\nexteq
\varphi(\alpha,\beta).
\end{align*}
\end{proof}

Proposition~\ref{lem:e3} shows that hypotheses $0<\deg_{x}\varphi<m_{1}$ and $0<\deg_{y}\varphi<n_{1}$ in Theorem~\ref{jisucan} are best possible, provided that any of the four equivalent conditions are satisfied. We conjecture that this is always the case.

\begin{conj}\label{acc}
Let $q$ be a prime power, and let $k,l$ be positive integers. Then 
there exists a monic irreducible polynomial $f(x)\in\F_{q^{k}}[x]$ of degree $l$ such that $f(x)-f(0)\in\F_{q}[x]$ and $f(0)\in\cF_{k}(q)$. 
\end{conj}
Note that Conjecture~\ref{acc} is slightly stronger than Proposition~\ref{lem:e3}(iv) in the sense that $l$ is not necessarily the smallest prime divisor of $kl$. 

\begin{conj}\label{acc2}
Let $p$ be a prime, and let $k,l$ be positive integers. Then 
there exists a monic irreducible polynomial $f(x)\in\F_{p^{k}}[x]$ of degree $l$ such that $f(x)-f(0)\in\F_{p}[x]$ and $f(0)\in\cF_{k}(p)$. 
\end{conj}

Clearly, validity of Conjecture~\ref{acc} for all prime power $q$
implies that of Conjecture~\ref{acc2}. 
Conversely, suppose that Conjecture~\ref{acc2} is true. 
Let $q=p^r$, where $p$ is a prime. 
Then there exists a monic irreducible polynomial $f(x)\in\F_{p^{rk}}[x]$ of degree $l$ such that $f(x)-f(0)\in\F_{p}[x]$ and $f(0)\in\cF_{rk}(p)$. In particular, $f(x)$ is a monic irreducible polynomial in $\F_{q^{k}}[x]$ of degree $l$ such that $f(x)-f(0)\in\F_{q}[x]$ and $f(0)\in\cF_{k}(q)$ by \eqref{honghai}.
Therefore, the two conjectures are equivalent.

The existence problem of a monic irreducible polynomial
of two prescribed coefficients dates back to Carlitz~\cite{C}. 
See \cite[Part~II, Section~3.5]{H} for more recent work.
Conjecture~\ref{acc2} is a similar but different problem,
in the sense that all coefficients except the constant term
are required to be in the prime field.

Conjecture~\ref{acc2} is trivially true for $l=1$ or $k=1$.
Moreover, it is true for the following special cases:

\begin{proposition}\label{prop:C1}
Conjecture~\ref{acc2} is true if $l=p$.
\end{proposition}
\begin{proof}
It is known (see for example \cite{C}) that there exists 
$a\in\cF_{k}(p)$ such that $\Tr_{\F_{p^{k}}}(a)=1$.
Then, by \cite[Corollary 3.79]{FF}, 
$x^{l}-x-a$ is irreducible in $\F_{p^{k}}[x]$.
\end{proof}

\begin{proposition}\label{prop:C2}
Let $l$ be a positive integer each of whose prime factor 
divides $p^k-1$. Assume further that,
$p^k\equiv1\pmod4$ if $l\equiv0\pmod4$.
Then Conjecture~\ref{acc2} is true.
\end{proposition}
\begin{proof}
Let $a$ be a primitive element of $\F_{p^k}$.
Then $x^l-a$ is irreducible in $\F_{p^k}[x]$ by
\cite[Theorem 3.75]{FF}.
\end{proof}

By Propositions~\ref{prop:C1} and \ref{prop:C2}, 
Conjecture~\ref{acc2} is true for $l=2$,
or $l=3$ and $k$ even.
We have verified Conjecture~\ref{acc2}
for $p^{kl}\leq 10^{20}$ by computer.

\end{document}